\documentclass[11pt,letterpaper]{article}
\usepackage{amsfonts, amsmath, amssymb, amscd, amsthm, graphicx}

\makeatletter
\def\blfootnote{\xdef\@thefnmark{}\@footnotetext}
\makeatother

\newtheorem{thm}{Theorem}[section]
\newtheorem{cor}[thm]{Corollary}
\newtheorem{lem}[thm]{Lemma}

\newtheorem{prob}[thm]{Problem}

\theoremstyle{definition}

\theoremstyle{remark}

\newfont{\eufm}{eufm10}

\newcommand{\Ker }{{\rm Ker }}
\newcommand{\e }{\varepsilon }

\newcommand{\rg }{{\rm RG}}

\renewcommand{\phi }{\varphi}

\renewcommand{\ll }{\langle\hspace{-.7mm}\langle }
\newcommand{\rr }{\rangle\hspace{-.7mm}\rangle }
\renewcommand{\d }{{\rm def} }

\begin{document}

\title{Rank gradient and torsion groups}
\author{D. Osin \thanks{This work has been partially supported by the NSF grant DMS-0605093. }}
\date{}
\maketitle

\begin{abstract}
We construct first examples of infinite finitely generated
residually finite torsion groups with positive rank gradient. In
particular, these groups are non-amenable. Some applications to
problems about cost and $L^2$-Betti numbers are discussed.
\end{abstract}

\medskip

\blfootnote{\textbf{Keywords.} Torsion group, residually finite
group, rank gradient, amenability, $L^2$-betti number, cost of a
group.}

\smallskip

\blfootnote {\textbf{2000 Mathematical Subject Classification.}
20F69, 20F05, 20E26.}

\section{Introduction}

The main goal of this paper is to suggest a new construction of infinite finitely generated residually finite torsion groups. First examples of such groups were discovered by Golod \cite{Golod} and many other constructions have been found since then \cite{Al,Gri,GS,OO}. Our approach is similar to that from \cite{OO}. Although it is quite elementary, it allows us to construct first examples of infinite finitely generated residually finite torsion groups with positive rank gradient.

Recall that the {\it rank gradient} of a finitely generated group $G$ is defined by the formula
$$
\rg (G)=\inf\limits_{\scriptstyle \left[ G:H\right] <\infty }\frac{d(H)-1}{[G:H]} ,
$$
where $d(H)$ denotes the minimal number of generators of $H$ and the infimum is taken over all finite index subgroups $H\le G$. This notion was first introduced by Lackenby in \cite{L05} with motivation from $3$-dimensional geometry and has received lots of attention since then.

Note that $\rg (F_m)=m-1$ for a free group $F_m$ of rank $m$ by
the Nielsen-Schreier Formula.  In \cite{L05}, Lackenby proved that
$\rg (G)>0$ for every free product $G=A\ast B$, where at least one
of the multiplies is not isomorphic to $\mathbb Z_2$. On the other
hand, many groups have zero rank gradient. This is easy to verify
for $SL_n(\mathbb Z)$ if $n\ge 3$, ascending $HNN$-extensions, and
direct products of finitely generated infinite residually finite
groups \cite{AN07a,L05}. Further $\rg (G)=0$ whenever $G$ is
amenable. This was first proved by Lackenby \cite{L05} for
finitely presented groups and then by Abert and Nikolov in the
general case \cite{AN07a}.

Recall also that, given a (finite or infinite) sequence of primes $\Pi
=(p_i)$, a {\it $\Pi $-series} in a group $G$ is any chain of normal subgroups
\begin{equation}\label{ser}
G=G_0\rhd G_1\rhd \ldots ,
\end{equation}
such that  $G_{i-1}/G_i$ is either trivial or abelian of exponent
$p_i$. A group $G$ is {\it $\Pi $-graded} if $G$ admits a $\Pi
$-series $\{ G_i\} $ such that $\bigcap\limits_i G_i=\{ 1\} $. Note that all subgroups $G_i$ have finite index in $G$ whenever $G$ is finitely generated. In particular, every finitely generated $\Pi $-graded group is residually finite.

In \cite{OS}, examples of finitely generated infinite torsion $\Pi
$-graded groups were  constructed for any infinite sequence of
primes $\Pi $, but the approach used in \cite{OS} did
not allow to control rank gradient. In this paper we prove the following.

\begin{thm}\label{main}
For every infinite sequence of primes $\Pi $, there exist a finitely generated infinite torsion
$\Pi $-graded group with positive rank gradient.
\end{thm}

It is easy to show that if $\Pi =(p,p,\ldots )$, then every torsion $\Pi $-graded group is a $p$-group. Thus we obtain the following.

\begin{cor}\label{pgroup}
For every prime $p$, there exist a finitely generated infinite
residually finite $p$-group with positive rank gradient.
\end{cor}

Applying Theorem \ref{main} to a sequence $\Pi $ of pairwise distinct primes, one can also obtain examples of different nature (see Section 3). It is also worth noting that the finitely generated residually finite torsion $p$-groups constructed by Aleshin \cite{Al}, Grigorchuck
\cite{Gri} and Gupta-Sidki \cite{GS} have zero rank gradient since
they are amenable. On the other hand, we do not know whether
Golod-Shafarevich groups have positive rank gradient.

Theorem \ref{main} is partially motivated by the von Neumann-Day Problem in the class of residually finite groups. Recall that the original problem asks whether there exists a non-amenable group without non-abelian free subgroups. The affirmative answer was obtained by Olshanskii in \cite{Ols}. Other examples can be found in \cite{Ad,Gro,OS}. However the question remained open for residually finite groups until it was recently solved affirmatively by Ershov \cite{E}. The solution is based on an unexpected example of a Golod-Shafarevich group with property (T) whose construction involves quite complicated arguments. Combining Theorem \ref{main} and the theorem of Lackenby-Abert-Nikolov mentioned above allows us to recover Ershov's result.

\begin{cor}[Ershov]\label{Ersh}
There exists a non-amenable residually finite group without free subgroups.
\end{cor}

The notion of rank gradient is closely related to some other group
invariants. Indeed for every finitely generated residually finite
group $G$, we have
\begin{equation}\label{ineq}
\rg (G)\ge {\rm cost} (G) -1 \ge \beta ^{(2)}  _1(G)-\frac{1}{|G|},
\end{equation}
where ${\rm cost} (G)$ denotes  the cost of $G$, $\beta ^{(2)}
_1(G)$ is the first $L^2$-Betti number of $G$ (we refer to
\cite{G00,Luck} for the definitions), and $1/|G|=0$ if $G$ is
infinite. The first inequality in (\ref{ineq}) is due to Abert and
Nikolov \cite{AN07b} and the second one is due to Gaboriau
\cite{G02}. In fact, in all cases where the the exact values of
these quantities are calculated they are equal, and it is still
unknown whether the inequalities in (\ref{ineq}) can be strict. The
following two questions were asked by Gaboriau \cite{G00,G02}.

\begin{prob}[Fixed Price]\label{c1}
Does every finitely generated group have fixed price?
\end{prob}

Recall that a countable group $G$ has fixed price,  if all
essentially free measure-preserving Borel actions of $G$ have the
same cost, which equals ${\rm cost} (G)$ in this case. Abert and
Nikolov \cite{AN07b} proved that for every finitely generated
residually finite group $G$, $\rg (G)$ equals the cost of the
natural action of $G$ on its profinite completion endowed with the
Haar measure minus $1$. Thus the Fixed Price Conjecture would imply
the equality $\rg (G)= {\rm cost} (G) -1$.

\begin{prob}[First $L^2$-Betti Number vs. Cost]\label{c2}
Is ${\rm cost} (G) -1 = \beta ^{ (2)}  _1(G)$ for every finitely generated infinite group $G$?
\end{prob}

The first $L^2$-Betti number of a finitely presented residually
finite group  $G$ can be defined in purely group theoretic terms due
to the following Approximation Theorem of Luck \cite{Luck}. Let $\{
N_i\} $ be a nested sequence of finite index normal subgroups of $G$
with trivial intersection. Then we have
\begin{equation}\label{b}
\beta _1^{(2)} (G)= \lim\limits_{i\to \infty} \frac{b_1(N_i)}{[G:N_i]},
\end{equation}
where $b_1(N_i)$ is the ordinary first  Betti number of $N_i$.
However the answer to the following question is still unknown.

\begin{prob}[Approximation for Finitely Generated Groups] \label{c3}
Does the approximation (\ref{b}) hold for every finitely generated residually finite group $G$?
\end{prob}

It is easy to see that Theorem \ref{main} implies the following.

\begin{cor}
At least one of the Problems \ref{c1}-\ref{c3} has negative solution.
\end{cor}

Indeed let $G$ be a finitely generated residually finite torsion group with positive rank gradient. We obviously have $b_1(N)=0$ for every subgroup $N\le G$ and hence the right side of (\ref{b}) equals zero. Thus either (\ref{b}) is violated or at least one of the inequalities in (\ref{ineq}) is strict.

{\bf Acknowledgments. } I am grateful to Miklos Abert and Alexander
Olshanskii for useful discussions. I am also grateful to Misha
Ershov who suggested Lemma \ref{misha} and noticed that it can be used to show that finitely generated infinite residually finite $p$-groups constructed by our method have positive rank gradient. This was overlooked by the author in the first version of the paper.

\section{Preliminaries}
In this section we collect some auxiliary definitions and results used in the proof of our main theorem.
Given a group $G$ and elements $x,y\in G$, we  write $x^y$ for $y^{-1}xy$. We
denote by $\ll S\rr ^G$ (or just $\ll S\rr $ if no confusion is
possible) the normal closure of a subset $S$ in $G$, i.e., the
smallest normal subgroup of $G$ containing $S$. By $d(G)$ we denote the minimal number of generators of a finitely generated group $G$. Finally if $G$ is finitely presented, $\d (G)$ denotes the deficiency of $G$.

Given a finitely presented group $G$,  its {\it deficiency} is defined as the maximum of $d-r$ over all finite
presentations
\begin{equation}\label{pres}
G=\langle x_1, \ldots , x_d\mid R_1, \ldots , R_r\rangle .
\end{equation}

We start with an obvious observation. It is an immediate corollary of the Neilsen-Schreier
formula.

\begin{lem}\label{NS}
Let $G$ be a finitely generated group, $K\le H$ finite index
subgroups of $G$. Then $(d(K)-1)/[G:K]\le (d(H)-1)/[G:H] $.
\end{lem}

The next lemma is also well-known. We outline the proof for the
sake of completeness.

\begin{lem}\label{def}
Let $G$ be a finitely presented group,  $H$ a finite index
subgroup of $G$. Then $\d (H) -1\ge (\d (G) -1)[G:H]$.
\end{lem}

\begin{proof}
Let $G=F/R$,  where $F=\langle x_1, \ldots , x_d\rangle $ is free
of rank $d$, $R=\ll R_1, \ldots , R_r\rr ^F$, and $\d (G) = d-r$.
By the Nilsen-Schreier Formula, the full preimage $K$ of $H$ in
$F$ has rank $(d-1)j+1$, where $j=[F:K]=[G:H]$. It is
straightforward to check that $R=\ll R_i^t, \, i=1, \ldots, r,\,
t\in T\rr ^K$, where $T$ is a set of left coset representatives of
$K$ in $F$. Thus $H=K/R$ has a presentation with $(d-1)j+1$
generators and $r|T|=rj$ relations, which implies $$\d (H)-1\ge
(d-1)j-rj=(d-r-1)j=(\d (G)-1)[G:H].$$
\end{proof}

\begin{lem}\label{quot}
Let $G$ be a finitely presented group, $N$ a finite index normal subgroup of $G$, $g$ an element of $G$, $m$ the order of $gN$ in $G/N$. Let $M$ denote the natural image of $N$ in the quotient group $Q=G/\ll g^m\rr $. Then $\d (M)\ge \d (N) - [G:N]/m $.
\end{lem}

\begin{proof}
Let $T$ denote the set of representatives of right cosets of $\langle g\rangle N$
in $G$. It is an easy exercise to prove that $\ll g\rr ^G= \ll Z\rr ^N$, where $Z=\{ g^{t} \, |\, t\in T\}$ (see, e.g., \cite{OS}[Lemma 2.3]).
Thus given a presentation of $N$, we can get a presentation of $M$ by adding $|T|=[G:\langle g\rangle N]=[G/N :\langle g\rangle N/N ]=[G:N]/m $ relations. This implies the statement of the lemma.
\end{proof}

To each group $G$ and each sequence of primes $\Pi =(p_i)$, we associate a sequence of subgroups $\{ \delta ^\Pi _i (G) \} $ of $G$ defined by $\delta ^\Pi _{0} (G)=G$ and $$\delta ^\Pi _i (G) =[\delta ^\Pi _{i-1} (G) , \delta ^\Pi _{i-1} (G)] \left(\delta ^\Pi _{i-1} (G)\right)^{p_i}.$$ It is
easy to prove by induction that these subgroups form a $\Pi $-series. Thus a group $G$ is $\Pi $-graded if and only if
\begin{equation}\label{int}
\bigcap_{i=1}^\infty \delta ^\Pi _{i} (G) =\{ 1\} .
\end{equation}

The proof of the following lemma is elementary and can be found in
\cite[Lemma 3.3 a)]{OO}.

\begin{lem}\label{ord}
Let $\Pi = (p_1, p_2, \ldots )$ be an infinite sequence of primes,
$G$ a torsion $\Pi $-graded group. Then for every $n\in \mathbb N$,
every element of $\delta _n^\Pi (G)$ has order $p_{n+1}^{\alpha
_{n+1}}\cdots p_{n+k}^{\alpha _{n+k}}$ for some integers $\alpha
_{n+1}, \ldots , \alpha_{n+k}$ and  $k>0$.
\end{lem}

\begin{cor}\label{csp}
Let $\Pi $ be an infinite sequence of pairwise distinct primes, $G$
a $\Pi $-graded torsion group. Then every finite index subgroup
$H\le G$ contains $\delta ^\Pi _{n} (G)$ for some $n$.
\end{cor}

\begin{proof}
Since every finite index subgroup contains a finite index normal
subgroup, we may assume that $H$ is normal without loss of
generality. Let $\Pi =(p_1, p_2, \ldots )$. As all primes in $\Pi $
are distinct, there exists an integer $n$ such that $|G/H|$ is not
divisible by $p_j$ for any $j> n$. Let $x$ be an arbitrary element
of $\delta _n^\Pi (G)$. By Lemma \ref{ord} we have ${\rm gcd} (|x|,
|G/H|)=1$ and hence the image of $x$ in $G/H$ is trivial. Thus we
have $\delta _n^\Pi (G)\le H$.
\end{proof}

The following analogue of Corollary \ref{csp} for $\Pi =(p,p,
\ldots)$, where $p$ is a prime, is quite trivial.

\begin{lem}\label{misha}
Let $\Pi =(p, p, \ldots )$, where $p$ is a prime, and let $G$ be a
torsion $\Pi $-graded group. Then every finite index subgroup of $G$
contains $\delta ^{\Pi }_n (G)$ for some $n$.
\end{lem}

\begin{proof}
Let $H\le G$, where $G$ is $\Pi$-graded and $H$ is of finite index.
Again we may assume that $H$ is normal in $G$. Then $G$ is a
$p$-group according to Lemma \ref{ord} and hence so is $Q=G/H$. Since every finite $p$-group
is nilpotent, we have $\delta ^{\Pi }_1(Q)=[Q,Q]Q^p\ne Q$. Therefore
$\delta ^{\Pi }_n (Q)=1$ for some $n$. This implies $\delta ^{\Pi
}_n (G)\le H$.
\end{proof}

\section{Torsion groups with positive rank gradient}

Theorem \ref{main} is a particular case of the following.

\begin{thm}\label{main1}
For every finitely presented group $G$ of deficiency $\d (G)\ge 2$, every infinite sequence of
primes $\Pi =(p_1,p_2, \ldots )$, and every $\e>0$, there exists an
infinite $\Pi $-graded torsion quotient group $Q$ of $G$ such that
\begin{equation}
\rg (Q)\ge \d(G)-1 -\e .
\end{equation}
\end{thm}

\begin{proof}
We fix a sequence of primes $\Pi $ and $\e>0$. Let
$r=\d(G)-1-\e$. Without loss of generality we can assume that $\e < 1$ and hence $r>0$. We enumerate all elements of $G=\{ g_0=1, g_2,
\ldots \} $ and  proceed by induction.

Suppose that for some $k\ge 0$, we have already constructed a
sequence of epimorphisms
\begin{equation}\label{seq}
G=G_0\stackrel{\alpha _1}\to G_1\stackrel{\alpha _2}\to \ldots
\stackrel{\alpha _k}\to G_k
\end{equation}
and a sequence of positive integers $0=n_0<n_1<\ldots <n_k$. In what
follows, we denote $\delta ^\Pi _{n_i} (G_i)$ by $D_i$. Given a
group $H$, we define $\widehat H=H/ \bigcap\limits_{j=1}^\infty
\delta_j^\Pi (H)$. Assume that for every $i=0, \ldots , k$, the
following conditions hold.

\begin{enumerate}
\item[(a)] The natural image $g_i$ in $\widehat{G}_i$ has finite
order.

\item[(b)] $\d (D_i)-1>r [G_i:D_i]$.
\end{enumerate}
Further if $k>0$, we assume that
\begin{enumerate}
\item[(c)] $\Ker (\alpha _{i+1}) \le [D_{i},
D_{i}]D_{i}^{p_{n_i+1}}$ for every $i=0, \ldots , k-1$.
\end{enumerate}

The following observation will be used several times below. Suppose that the image of
an element $h$ of a group $H$ has finite order in $\widehat H$. Then for every
quotient group $K$ of $H$, the image of $h$ in $\widehat K$ has finite order. This follows immediately from the fact that $\delta ^\Pi _i(H)$ is a verbal subgroup of $H$ and hence $\delta ^\Pi _i(H)=\delta ^\Pi _i(K)$ for all $i$.

Let us consider two cases.

{\it Case 1.} Suppose that the natural image of $g_{k+1}$ in
$\widehat G_k$ has finite order. Then we set $G_{k+1}=G_k$,
$\alpha _{k+1}=id$, and $n_{k+1}=n_k+1$. Condition (b) for $i=k+1$ easily
follows from the inductive assumption and Lemma \ref{def} applied to the groups $D_{k+1}\le D_k$.
Verifying (a) and (c) is straightforward.

{\it Case 2.} Suppose now that the natural image of  $g_{k+1}$ in
$\widehat G_k$ has infinite order. This means that the order of
the image of $g_{k+1}$ in $G_k/\delta _j^\Pi (G_k)$ tends to
$\infty $ as $j\to \infty $. Therefore by (b) we can find
$n_{k+1}>n_k$ such that such that
\begin{equation}\label{ddk}
\frac{\d (D_{k}) -1}{[G_k:D_{k}]}>r +\frac{1}{m},
\end{equation}
where $m$ is the order of the image of $g_{k+1}$ in
$G_k/\delta^\Pi _{n_{k+1}}(G_k)$.  We set $G_{k+1}=G_k/\ll \bar
g_{k+1}^m\rr^{G_k}$, where $\bar g_{k+1} $ is the image of
$g_{k+1}$ in $G_k$, and let $\alpha _{k+1}$ be the natural
homomorphism $G_k\to G_{k+1}$.

Again verifying (c) for $i=k$ is trivial. Indeed we have  $$\Ker (\alpha _{k+1}) = \ll \bar
g_{k+1}^m\rr^{G_k} \le \delta_{n_{k+1}}^\Pi (G_{k})\le \delta_{n_{k}+1}^\Pi (G_{k})=[D_{k},
D_{k}]D_{k}^{p_{k+1}} $$ since $\bar
g_{k+1}^m \in \delta^\Pi _{n_{k+1}}(G_k)$ by our construction.

It is also straightforward to verify (a) for $i=k+1$.
Finally to prove (b) we note that $D_{k+1}=\delta_{n_{k+1}}^\Pi
(G_{k+1})=\alpha _{k+1} (\delta_{n_{k+1}}^\Pi (G_{k}))$. Using subsequently
Lemma \ref{quot}, Lemma \ref{def}, (\ref{ddk}), and the inclusion $\Ker (\alpha _{k+1})\le \delta_{n_{k+1}}^\Pi (G_{k})$, we
obtain
$$
\begin{array}{rl}
\d (D_{k+1}) -1 \ge & \d (\delta_{n_{k+1}}^\Pi (G_{k}))-1 -
\frac{1}{m}[G_{k}:\delta_{n_{k+1}}^\Pi (G_{k})]  \ge \\ &
\\  & (\d (D_k)-1)[D_k:\delta_{n_{k+1}}^\Pi (G_{k})] - \frac{1}{m}[G_{k}:\delta_{n_{k+1}}^\Pi (G_{k})]> \\ & \\
& (r+\frac1m)[G_k:D_k][D_k:\delta_{n_{k+1}}^\Pi (G_{k})]
-\frac{1}{m}[G_{k}:\delta_{n_{k+1}}^\Pi (G_{k})] = \\ & \\ &
r[G_{k}:\delta_{n_{k+1}}^\Pi (G_{k})] = \\ & \\ & r[\alpha _{k+1}(G_{k}):\alpha _{k+1}(\delta_{n_{k+1}}^\Pi (G_{k}))]= \\ & \\ & r[G_{k+1}:D_{k+1}].
\end{array}
$$
This finishes the inductive step.

Let now $P$ be the limit group. That is, let
$P=G_0/\bigcup_{i=1}^\infty K_i$, where $K_i={\rm Ker}\, (\alpha
_i\circ \cdots \circ \alpha _1) $. Further let $Q=\widehat P$.
Clearly $\bigcap\limits_{i=1}^\infty \delta _i^{\Pi } (Q)=\{ 1\} $
and hence $Q$ is a $\Pi $--graded group. It follows from (a) that every element of $Q$ has finite
order.

Let $$A_{i}=\delta _{n_i}^{\Pi } (Q)/[\delta _{n_i}^{\Pi } (Q), \delta _{n_i}^{\Pi } (Q)](\delta _{n_i}^{\Pi } (Q))^{p_{n_i+1}}.$$
Observe that for every $i\in \mathbb N$, $A_{i}$ is isomorphic to $D_i/[D_i,D_i]D_i^{p_{n_i+1}}$ as $\Ker (G_i\to Q)\le [D_i,D_i]D_i^{p_{n_i+1}}$ by
(c). Hence $A_{i}$ is isomorphic to a quotient group of $(\mathbb
Z/\mathbb Z_{p_{n_i+1}})^{d_i}$ modulo a subgroup of rank $r_i$, where
$d_i-r_i =\d (D_i)$. In particular, $d (A_{i})\ge \d (D_i)$ and using (b) we obtain

\begin{equation}\label{rgQ}
d(\delta _{n_i}^{\Pi } (Q))-1\ge d(A_{i})-1\ge \d (D_i)-1> r [G_i:D_i]=r[Q:\delta _{n_i}^{\Pi } (Q)]
\end{equation}
for every $i\in \mathbb N$. Again we consider two cases.

{\it Case 1.} Suppose that $\Pi $ contains infinitely many distinct
primes. Note that if a group is $\Xi $-graded for some subsequence $\Xi $ of $\Pi $, then it is $\Pi $-graded as well. Thus passing to a subsequence of $\Pi $ if necessary we can assume that $\Pi $ consists of pairwise distinct primes. Let $Q$ be an
infinite $\Pi $-graded torsion quotient group of $G$ satisfying (\ref{rgQ}). For every finite index
subgroup $H\le Q$, we have $(d(H)-1)/[Q:H]\ge r$ by Corollary
\ref{csp}, Lemma \ref{NS}, and (\ref{rgQ}). Thus $\rg (Q)\ge r=\d
(G)-1-\e $.

{\it Case 2.} If $\Pi $ contains finitely many distinct primes, then at least one prime $p$ occurs in $\Pi $ infinitely many times. Thus passing to a subsequence if necessary we can assume that $\Pi =(p,p, \ldots )$. The rest of
the proof in this case is the same as above. The only difference is
that we have to use Lemma \ref{misha} instead of Corollary
\ref{csp}.
\end{proof}

\begin{proof}[Proof of Corollary \ref{pgroup}]
Let $\Pi = (p,p, \ldots )$. By Theorem \ref{main1} there exists a
finitely generated residually finite infinite $\Pi $-graded torsion
group $Q$ with positive rank gradient. By Lemma \ref{ord} $Q$ is a
$p$-group.
\end{proof}

Corollary \ref{Ersh} is a particular case of the following.

\begin{cor}
For every infinite sequence of primes $\Pi $,  there exists a
finitely generated torsion non-amenable $\Pi $-graded group $G$.
In particular, for every  prime $p$, there exists a finitely
generated residually finite non-amenable $p$-group.
\end{cor}

\begin{proof}
By the Lackenby-Abert-Nikolov Theorem \cite{AN07a}, the groups constructed in Theorem \ref{main1} are non-amenable if $\e <1$.
\end{proof}

\end{document}